\documentclass[12pt]{extarticle}
\usepackage[utf8]{inputenc}
\usepackage{appendix}
\usepackage{amsfonts}
\usepackage{amsthm}
\usepackage{amsmath}
\usepackage{mathrsfs} 
\usepackage{ amssymb }
\usepackage[margin=1.2in,footskip=0.5in]{geometry}
\usepackage{hyperref}  
\usepackage{cleveref}
\usepackage{graphicx}
\usepackage{tikz}
\usepackage[all]{xy}
\usetikzlibrary{matrix,arrows,positioning, decorations.pathmorphing,backgrounds,fit,petri,calc}
\usepackage{tkz-euclide}
\usepackage{tikz-cd}
\usepackage{float}
\newtheorem{theorem}{Theorem}[section]
\newtheorem{proposition}[theorem]{Proposition}
\newtheorem{definition}[theorem]{Definition}
\newtheorem{corollary}[theorem]{Corollary}

\newtheorem*{theorem*}{Theorem}

\newtheorem{remark}[theorem]{Remark}
\newtheorem{lemma}[theorem]{Lemma}
\newtheorem*{proposition*}{Proposition}

\newcommand{\C}{\mathbb C}
\newcommand{\Hom}{\mathrm{Hom}}
\renewcommand{\P}{\mathbb P}

\newcommand{\NN}{\widetilde{\mathcal N}}

\newcommand{\F}{\mathscr F}
\newcommand{\g}{\mathfrak{g}}
\renewcommand{\b}{\mathfrak{b}}

\renewcommand{\O}{\mathcal O}
\newcommand{\n}{\mathfrak{n}}
\renewcommand{\u}{\mathfrak{u}}

\newcommand{\Ext}{\text{Ext}}
\newcommand{\Spec}{\text{Spec}}
\newcommand{\Fgt}{\text{Fgt}}

\title{Twisted equivariant HKR theorem for torus action and the small quantum group}

\author{Nicolas Hemelsoet}

\begin{document}

\maketitle

\begin{abstract}
  
We show that when a torus $T$ acts on a smooth variety $X$, the twisted HKR isomorphism is equivariant. The main consequence is that the Bezrukavnikov-Lachowska isomorphism, relating the Hochschild cohomology of the principal block of the small quantum group to certain sheaf cohomology groups on the Springer resolution $\NN$, can be upgraded to a ring isomorphism by a twist.

\end{abstract}

\begingroup
\def\addvspace#1{}
\tableofcontents
\endgroup

\section{Introduction}


The starting point of this article is the following theorem :

\begin{theorem}[\cite{BL}]\label{thm0}
There exists an algebra isomorphism \[ HH^{\bullet}(\u_0(\g)) \cong \bigoplus_{i+j+k= \bullet} H^i(\NN, \wedge^j T \NN)^k \]
\end{theorem}

This theorem relates the Hochschild cohomology of the principal block of the small quantum group with an "equivariant" version of the Hochschild cohomology of the Springer resolution. More formally, let us detail the notation : $\g$ is a semisimple Lie algebra, $\u_0(\g)$ is the principal block of the small quantum group associated to $\g$, $\NN$ is the \emph{Springer resolution} $\NN := T^*(G/B) = G \times_B \n$ where $G$ is the connected simple algebraic group of adjoint type associated to $\g$, $B \subset G$ is a Borel subgroup and $\n = [\b, \b]$ where $\b = Lie(B)$. $HH^{\bullet}(\u_0(\g))$ is the Hochschild cohomology of $\u_0(\g)$. The subscript $k$ correspond to the grading acquired from the $\C^*$ action by dilations on the fiber on $\NN$. We refer to \cite{BL} and \cite{LQ} for a more detailled exposition. \\ 

The multiplication on the left-hand side is given by the Yoneda product, and on the right-hand side  by the exterior algebra structure. The purpose of this paper is to explain precisely how to obtain a multiplicative isomorphism, using a theorem by Kontsevitch. We will check the compatibility of the multiplicative version of the Hochschild-Kostant-Rosenberg theorem with a action of a torus. \\ 

Let $T$ be a complex torus, and $X$ be an irreducible smooth complex variety acted on by $T$. More generally, one can consider an arbitrary complex reductive group acting on a smooth variety $X$, if there exists an affine cover by invariant open subset. Let $HH^{\bullet}(X)$ be the Hochschild cohomology of $X$, defined as $Ext^{\bullet}_{Coh(X \times X)}(\O_{\Delta}, \O_{\Delta})$, where $\Delta \subset X \times X$ is the diagonal (for a variety $Z$ we denote $\O_Z$ the structure sheaf), and $HT^{\bullet}(X)$ denotes the bigraded vector space $\bigoplus_{i,j} H^i(X, \wedge^j TX)$. Our first result is : 
\begin{theorem}\label{thm1}
The twisted HKR morphism $HH^{\bullet}(X) \to HT^{\bullet}(X)$ is $T$-equivariant.
\end{theorem}

This result relies on the following two propositions :

\begin{proposition}\label{prop0}
In the derived category of coherent sheaves on $X$ $D^b(\text{Coh}(X))$, the quasi-isomorphism $$\iota^*\O_{\Delta} \cong \bigoplus_{i \in X} \Omega^i_X[i]$$ is $T$-equivariant. 
\end{proposition}

Here $\iota : \Delta \to X \times X$ is the inclusion map, $\Omega_X^i$ is the sheaf of differential forms of degree $i$, and $[i]$ is the shift functor in the derived category. 

\begin{proposition}\label{prop1}
Let $t \in H\Omega^{\bullet}(X)$ be the Todd class of $\NN$. Then, twisting with $t^{-1/2}$ gives a $T$-equivariant multiplicative isomorphism $HH^{\bullet}(X) \cong HT^{\bullet}(X) $.
\end{proposition}

When no group action is involved, the two statements are well-known (\cite{ginzburg}, section $9$). The first proposition implies the famous Hochschild-Kostant-Rosenberg (HKR) theorem relating the Hochschild cohomology of $X$ with sheaf cohomology of poly-vector fields. The second statement relates the multiplicative structures on both sides. Hence the essence of our results is the compatibility with the action of $T$. Using our first theorem, we can deduce an explicit way to make the Bezrukavnikov-Lachowska isomorphism multiplicative. Setting $X = \NN$ and $T = \C^*$, we will deduce in section $3$ the following theorem, describing geometrically the multiplicative structure of $z_0(\g)$ :

\begin{theorem}\label{thm2} 
The composition \[ HH^{\bullet}(\mathfrak{u}_0(\mathfrak{g})) \xrightarrow{HKR}  \bigoplus_{i+j+k= \bullet} H^i(\NN, \wedge^j T \NN)^k \xrightarrow{\langle -, Todd(\NN)^{-1/2} \rangle } \bigoplus_{i+j+k= \bullet} H^i(\NN, \wedge^j T \NN)^k  \] is a ring isomorphism.
\end{theorem}

Taking the zero degree part of the previous isomorphism gives the following corollary:

\begin{corollary}\label{cor2}
The natural bigraded vector space isomorphism \[ HH^0(\u_0(\g)) \cong \bigoplus_{i+j+k= 0} H^i(\NN, \wedge^j T \NN)^k \] can be upgraded to a multiplicative isomorphism using a twist.
\end{corollary}

The theorem \ref{thm2} and the corollary \ref{cor2} also hold for singular blocks, where one needs to twist by the corresponding ``parabolic Todd class". \\

Let us briefly outline the contents of this paper : in section $2$ we study equivariant sheaves on smooth varieties, give background material for the Kontsevitch theorem and prove theorem \ref{thm1}. In section $3$ we explain the connection with the small quantum group, and derive some consequences for the structure of $z_0(\g)$.

\subsection*{Acknowledgements}

I would like to thanks Anna Lachowska and Qi You for support and suggesting the project. I would also like to thanks Pieter Belmans, Damien Calaque, Dmitriy Rumynin and Andras Szenes for helpful and interesting discussions. 

\section{Geometry}

We first recall a theorem by Sumihiro :

\begin{theorem}\label{sumihiro}\cite{sumihiro}
Let $T$ be a torus and $X$ a complex $T$-algebraic variety. Then, there is an open cover $\mathfrak{U}$ of $X$, which is $T$-invariant (i.e for all $U \in \mathfrak{U}$, we have $TU \subset U$).
\end{theorem}

This theorem reduces many statements about algebraic varieties with a $T$-action to an affine statement about graded $R$-modules that are easier to prove. 

\subsection{Equivariant sheaves}

Let $X$ be a variety over the complex numbers, with an algebraic action of a complex reductive group $K$ (which is not related to the Lie algebra $\g$ mentioned at the beginning of the introduction). We will use $\rm{QCoh}^K(X)$ for the category of quasi-coherent $K$-equivariant sheaves. First we recall the definitions and standard facts of an equivariant sheaf : 

\begin{definition}
A sheaf $\F \in \rm{QCoh}(X)$ is $K$-equivariant if there is an isomorphism $\theta :m^*\F \cong pr_2^*\F$, where $m, pr_2 : K \times X \to X$ are respectively the action and the projection (moreover, $\theta$ should satisfies a cocycle condition).
\end{definition}

Taking the stalks of $\theta$ at $(g,x) \in K \times X$ gives an isomorphism $\theta_x : \F_{g \cdot x} \cong \F_x$. A  $K$-equivariant sheaf $\F$ is equivalent to the data of a sheaf $\F$ with such isomorphisms. We also have 

\begin{lemma}\label{eqsheaf}
A $K$-equivariant sheaf $\F$ is equivalent to a sheaf $\F$ with maps $g : \Gamma(U, \F) \to \Gamma(g(U), \F)$ for each $g \in T$ and $U$ open, commuting with restrictions and such that $g_1 \cdot (g_2 \cdot s) = (g_1 \cdot g_2) \cdot s$ for all $g_1, g_2 \in K$ and all $s \in \F(U)$. 
\end{lemma}

\begin{corollary}
There is a natural structure of $K$-modules on $\Gamma(X, \F)$ for any equivariant sheaf $\F$. More generally, there is a structure of $K$-modules on $Hom_{\rm{QCoh(X)}}(\F_1, \F_2)$ if $\F_i \in \rm{QCoh}^K(X)$.
\end{corollary}

\begin{proof}
The first point is clear. The action in general is defined as follow : if $\varphi : \F_1(U) \to \F_2(U)$, we define $g \cdot \varphi = g \circ \varphi \circ g^{-1}$, where $g$ denotes the map in lemma \ref{eqsheaf}.
\end{proof}

\begin{definition}
The category $\rm{QCoh}^K(X)$ is the category where objects are $K$-equivariant quasi-coherent sheaves of $\O_X$-modules, and morphisms are equivariant morphisms of $\O_X$-modules, i.e. the maps $f : \F_1 \to \F_2$ such that that $f(g \cdot s) = g \cdot f(s)$ for all $s \in \F_1(U)$ and all $g \in T$.
\end{definition} 

To give more intuition, we recall that there is an equivalence of categories between the category of locally free coherent $K$-equivariant sheaves and the category of finite-dimensional vector bundles with an action of $K$, linear on each fibers, where morphisms are $K$-equivariant morphisms of vector bundles (see \cite{CG}, chapter $5$). \\ 

\begin{lemma}\label{semisimple}
The $K$-module $V := Hom_{\rm{QCoh}(X)}(\F_1, \F_2)$ is semisimple. Moreover there is an isomorphism $Hom_{\rm{QCoh}^K(X)}(\F_1, \F_2) = (Hom_{\rm{QCoh}(X)}(\F_1, \F_2))^K$. 
\end{lemma}

\begin{proof}
The induced $K$-action is algebraic, i.e given by a ring map $V \to V \otimes \O(K)$. It follows easily that $V$ is the union of its finite-dimensional $K$-modules, and hence semi-simple. The second part follows from the definition of the $K$-action. 
\end{proof}

We now investigate basic properties of $K$-equivariant sheaves in order to be able to describe derived functors, now computed in the equivariant category. 

\begin{remark}
Let us emphasize that in geometric representation theory, the usual definition of the "equivariant derived category" (as defined e.g in Bernstein-Lunts :\cite{BernsteinLunts}) does not  coincide with the derived category of equivariant sheaves (the one we consider here). 
\end{remark}

\begin{lemma}
For each $\F\in \rm{QCoh}^K(X)$, there is an injective object $I \in \rm{QCoh}^K(X)$ and a $K$-equivariant monomorphism $\F \to I$. Moreover $\Fgt(I) \in \rm{QCoh(X)}$ is still injective, where $\Fgt : \rm{QCoh}^K(X) \to \rm{QCoh}(X)$ is the forgetful functor.
\end{lemma}

\begin{proof}
The first part is proposition 5.1.2 in \cite{grothendieck}. The second part follows from the explicit construction of $I$ in the same paper.
\end{proof}

This lemma ensure that we can consider right derived functors of $Hom(\F_1,-)$.

\begin{corollary}
Let $\F_1 \in \rm{Coh(X)}^K(X)$ and $\Gamma_1 = Hom_{\rm{QCoh^K(X)}}(\F_1,-) : \rm{Coh^K(X)} \to K-Mod$. Then there is a canonical isomorphism $R(\Gamma_1 \circ \Fgt) \cong R \Gamma_1 \circ \Fgt$.
\end{corollary}

\begin{corollary}\label{cor3}
For any equivariant sheaf $\F_1$, there is an isomorphism between the functors $\rm{QCoh}^K(X) \to \rm{Vect}$ : $$R(\rm{Inv} \circ  \Gamma_{\F_1}) \cong \rm{Inv} \circ R\Gamma_{\F_1} $$
where $\Gamma_{\F_1} := Hom_{\rm{QCoh(X)}}(\Fgt(\F_1),\Fgt(-))$, and $Inv : K-Mod \to \rm{Vect}$, $Inv(M) = M^K$. In particular if $\F_1,\F_2$ are $K$-equivariant sheaves on $X$, there is an isomorphism $$Ext^q_{\rm{QCoh}^K(X)}(\F_1, \F_2) \cong (Ext^q_{\rm{QCoh}(X)}(\F_1, \F_2))^K$$
\end{corollary}

\begin{proof}
Recall the Grothendieck formula to derive compatible left-exact functors : $R(F \circ F') \cong RF \circ RF'$. Taking invariant is exact on the subcategory of semisimple $K$-representations, and by lemma \ref{semisimple} $\Gamma(Hom_{\rm{QCoh(X)}}(\F_1, \F_2))$ is semisimple. Hence $R \text{Inv} = \text{Inv}$ and the formula follows. 
\end{proof}

\begin{corollary}\label{cor1}
If $\F$ is a $K$-equivariant sheaf, then there is an isomorphism $$ RHom_{\rm{QCoh^K(X)}}(\O_X, \F) \cong (RHom_{\rm{QCoh(X)}}(\O_X, \F))^K$$
\end{corollary}

\begin{lemma}\label{inv}
If $0 \to \F_1 \to \F_2 \to \F_3 \to 0$ is a $K$-equivariant sequence of $K$-equivariant sheaves, and $\F_3$ is locally free, then the corresponding cohomology class in $H^1(X, \F_1 \otimes \F_3^*)$ is $K$-invariant.
\end{lemma}

\begin{proof}

By definition the short exact sequence gives a map $\F_3 \to \F_1[1]$ in $D^b(\rm{QCoh}^K(X))$. This is the same as an element in $Ext^1_{\rm{QCoh}^K(X)}(\F_3, \F_1)$. Since $ Ext^1_{\rm{QCoh}^K(X)}(\F_3, \F_1) \cong Ext^1_{\rm{QCoh}^K(X)}(\O_X, \F_1 \otimes \F_3^*)$  we can apply the corollary \ref{cor3} and compute the latter as $H^1(R\Gamma(X,\F_1 \otimes \F_3^*)^K)$ which is $K$-invariant.
\end{proof}

Now we specialize to the case where $K = T$ is a complex algebraic torus. 

\begin{lemma}\label{trivial}
Assume that the $T$-action on $X$ is trivial and that $\mathcal E$ is a $T$-equivariant vector bundle on $X$. Then, there is a decomposition into eigenspace $\mathcal E = \oplus_{\lambda} \mathcal E[\lambda]$ where $\lambda : T \to \C^*$ is a character and we have $H^i_{T}(X, \mathcal E) = H^i(X, \mathcal E[0])$.
\end{lemma}

\begin{proof}
If $U \subset X$ is affine, then $\mathcal E_{\mid U} \cong E \times U$ for some vector space $E$. The action of $T$ is given by an action on the first component, and we get a decomposition $E \cong \oplus_{\lambda} E[\lambda]$ into $T$-eigenspaces. Hence $\mathcal E_{\mid U} \cong \oplus_{\lambda} \mathcal E_{\mid U} [\lambda]$, and we get a global decomposition $\mathcal E \cong \oplus_{\lambda} \mathcal E[\lambda]$ since the $T$ -action does not depend on $U$. Moreover, there is an isomorphism of functors $\Gamma_T(X,\mathcal E) \cong \Gamma(X, \mathcal E[0])$. It follows that $H^i_{T}(X, \mathcal E) = H^i(X, \mathcal E[0])$ by taking derived functors. 
\end{proof}

Most of these results were used implicitly in \cite{BL}.
 
\subsection{Kontsevitch's theorem}

Let $X$ be a smooth algebraic variety over the complex numbers. The Hochschild cohomology of $X$ is defined as $HH^{\bullet}(X) = \bigoplus_{i}\Ext^i_{\O_{X \times X}}(\O_{\Delta}, \O_{\Delta})$ where $\Delta \subset X \times X$ is the diagonal. It contains information about deformations of $X$, for example $H^1(X, TX)$ is the space parametrizing complex deformations of $X$. Let us recall two important theorems about the Hochschild cohomology of a smooth complex algebraic variety. The first theorem is the \emph{Hochschild-Kostant-Rosenberg isomorphism }

\begin{theorem}(\cite{ginzburg},\label{qis})
Let $\iota : \Delta \to X \times X$ be the diagonal inclusion.  There is a quasi-isomorphism $$\iota^*\O_{\Delta} \cong \bigoplus_{i \in \mathbb N} \Omega_X^i[i] $$ 
in the derived category $D^b(\text{Coh}(X))$.
\end{theorem}

Here $\Omega_X^i$ is the sheaf of differential forms on $X$ of degree $i$, and $[i]$ is the cohomological shift functor. \\ 

This is a stronger version of the HKR theorem (that appeared first in \cite{hochschild}) stating that there is an isomorphism of graded vector spaces $$I_{HKR} :  HH^{\bullet}(X) \to HT^{\bullet}(X) := \bigoplus_{i + j = \bullet} H^i(X, \wedge^jTX))$$

Let us recall the proof of theorem \ref{qis} : a non-trivial argument using sheaves (see \cite{yekutieli}, \cite{ginzburg}) reduces it to the statement when $X = \Spec(R)$ is affine. This is the case treated in \cite{hochschild}. We use the \emph{bar complex} $\mathscr{C}_{\bullet}(R)$, which is a flat resolution of $\O_{\Delta}$ by $\O_{X \times X}$-modules. Recall that the bar resolution is given by $\mathscr{B}_i(R)  := R^{\otimes (i+2)} $. We define $\mathscr{C}_i(R) = R \otimes_{R \times R} \mathscr{B}_i(R)$. The differential on $\mathscr{B}_{\bullet}(R)$ is given by $$d(1 \otimes a_1 \otimes \dots \otimes a_i \otimes 1) = a_1 \otimes a_2 \dots \otimes a_i \otimes 1 - 1 \otimes a_1a_2 \otimes a_3 \otimes \dots \otimes a_i \otimes 1 + \dots + (-1)^{i+1}1 \otimes a_0 \otimes \dots \otimes a_i$$ A quasi-isomorphism  $I : \mathscr{C}_{\bullet}(R) \to (\bigwedge^{\bullet}(\Omega_R^1[1]),0)$ is given by $$1 \otimes (1 \otimes a_1 \otimes \dots \otimes a_i \otimes 1) \mapsto da_1 \wedge da_2 \dots \wedge da_i$$

Since the bar complex resolves $R$ as $R$-bimodule, this bar complex is isomorphic to $\O_{\Delta}$
in $D^b(X \times X)$. This finishes the proof.

In general, $I_{HKR}$ is not multiplicative. To understand how the multiplications are related, we need to introduce the Todd class. We use the formalism of Atiyah class explained in section $1$ of \cite{kapranov}.

There is an exact sequence $$ 0 \to \Omega_X \otimes TX \to J^1(TX) \to TX \to 0 $$
where $J^1(TX)$ is the bundle of first-order jet of $TX$. Such sequence gives a class in $Ext^1(TX, \Omega_X \otimes TX) = H^1(X, \Omega_X \otimes End(TX))$. 

\begin{definition}
The \emph{Atiyah class} of $X$ is defined as the class of this extension : $At(X) \in H^1(X, \Omega_X \otimes \text{End}(TX))$.
\end{definition}

\begin{definition}
The \emph{Todd class} of $X$ is defined as \[ Td(X) = \det \left( \frac{At(X)}{1 - e^{-At(X)}} \right) \]
\end{definition}

A more familiar definition of the Todd class involves Chern classes of $TX$. Our definition is equivalent to that, according to the following statement (see \cite{kapranov}, formula $1.4.1$) :

\begin{proposition}
We have $tr(\wedge^i At(X)) = c_i(X)$ where $c_i(X)$ is the $i$-th Chern class of the tangent bundle $TX$.
\end{proposition}

Now we can state Kontsevitch's theorem :

\begin{theorem}\cite{kontsevitch}\cite{calaque}
Let $X$ be as before and $t$ be the Todd class of $X$. Then, the composition $$ \langle t^{-1/2},- \rangle \circ  I_{HKR}  : HH^{\bullet}(X) \to HT^{\bullet}(X)$$ is a ring isomorphism.
\end{theorem}

\subsection{Equivariance of the twisted HKR isomorphism}

We now assume that $T$ is a complex torus acting on a smooth variety $X$.

\begin{lemma}\label{lemme1} If $X$ is affine, the isomorphism $I_{HKR}$ is $T$-equivariant.
\end{lemma}

\begin{proof} 
Say $X = \Spec(R)$. Let us keep notation from the discussion after \ref{qis}. The differential of the bar-complex is $T$-equivariant, hence the bar-complex gives a flat $T$-equivariant resolution of $\O_{\Delta}$. Moreover the map $I : A \to \Omega_R, a \mapsto da$ is $T$-equivariant, where $\Omega_R$ is the module of Kähler differential. This is because there is a natural identification $\Omega_R \cong J/J^2, da \mapsto a \otimes 1 - 1 \otimes a$, where $J$ is the kernel of the multiplication map $R \otimes R \to R$.  Under this identification, the map $a \mapsto da$ is $T$-equivariant. It follows that the quasi-isomorphism $I : \mathscr{C}_{\bullet}(R) \to (\bigwedge^{\bullet}(\Omega_R^1[1]),0)$ is also $T$-equivariant. 
\end{proof}

\begin{lemma}\label{lemme2}
If $X$ is a quasi-projective smooth variety acted upon by $T$ then the previous lemma holds. 
\end{lemma}

\begin{proof} 
By Sumihiro's theorem, there is an $T$-invariant affine open cover. So we can check the statement on each affine open set $U \subset X$ that is $T$-invariant. But recall that in the discussion sketching the proof of theorem \ref{qis}, we saw that we could glue together the bar complex to construct a global quasi-isomorphism $I_{HKR}$. So it means that the invariance can be only checked for the bar complex $\mathscr C_{\bullet}(R)$, that was precisely given by our previous lemma. Hence $I_{HKR}$ is $T$-equivariant. 
\end{proof}

\begin{lemma}\label{lemme3}
If $X$ is as before, then the Atiyah class of $X$ is $T$-invariant. Therefore the Todd class of $X$ is also $T$-invariant.
\end{lemma}

\begin{proof}
This follows from lemma \ref{inv} because the exact sequence defining the Atiyah class is $T$-equivariant.
\end{proof}

It follows that the twisted HKR morphism $HH^{\bullet}(X) \to HT^{\bullet}(X)$ is $T$-equivariant, which completes the proof of our theorem \ref{thm1}. It would be interesting to generalize theorem \ref{thm1} and relate it with the geometry of $X/G$, where $G$ is a complex reductive group. However it is known that even when $G$ is a finite group, the analogue of Kontsevitch's theorem does not hold for the stack $[X/G]$, see \cite{NS}. Without Sumihiro's theorem,(i.e without the existence of an invariant cover) it seems non-trivial to prove that $I_{HKR}$ is $G$-equivariant.

\section{Applications to the small quantum group}

We recall notations and results from \cite{BL} and prove our main result, theorem \ref{thm2}. As an application, we show that the multiplication on the subalgebra generated by the Harish-Chandra center and the Poisson bi-vector field $\tau$ is untwisted.

\subsection{The small quantum group}

We summarize results from \cite{BL}. Let $G$ be a semisimple Lie group of adjoint type, $H \subset B$ a Cartan subgroup contained in a Borel subgroup and $N$ the unipotent radical of $B$. We write $\g, \b, \n$ for the corresponding Lie algebras. If $q$ is a primitive $\ell$-th root of unity where $\ell$ is odd, greater that the Coxeter number of $\g$, prime to the index of connection, and prime to $3$ if there is a component of type $G_2$. To this data Lusztig associated a canonical finite-dimensional Hopf algebra $\u_q(\g)$, called the \emph{small quantum group}, see e.g \cite{lusztig}. An important problem in representation theory is to understand its center. Since $\u_q(\g)$ splits as sum of two-sided ideals, \emph{blocks}. They are parametrized by the orbit of a certain extension of the Weyl group in the weight lattice. So it is enough to understand structure of the center for each block. The regular block that contains the trivial representation is the \emph{principal block} (corresponding to $\lambda = 0$), denoted by $\u_0(\g)$ or $\u_0$ for simplicity. 

A major advance was obtained in \cite{BL} where the center was described geometrically via the Springer resolution. Before stating the main result, let us introduce some notations : $z$ is the trivial line bundle on $\NN$ with the $\C^*$-action given by $(z,v) \mapsto z^2v$, $R$ is the principal block, viewed as bimodule over $\u$. By definition $H^i(\NN, (\wedge^jT\NN)^k) := H^i(G/B, (pr_* \wedge^jT \NN)^k)$.  By definition $HH^{\bullet}(\u_0)$ is the graded ring $$\bigoplus_{m} \Hom_{D^b(\u_0 \otimes \u_0^{op})}(R,R[m])$$ where the multiplication is given by the composition. Moreover, $HH^*_{\C^*}(\NN)$ is by definition the ring $$\bigoplus_{j+k=m} \Hom_{D^b(\mathrm{Coh}^{\C^*}(\NN \times \NN))} (\O_{\Delta}, \O_{\Delta}[j] \otimes z^k)$$ where the ring structure comes from the Yoneda product. The main theorem in \cite{BL} can be formulated :

\begin{theorem}\label{thmBL}
There is a ring isomorphism $ HH^{\bullet}(\u_0) \cong HH^*_{\C^*}(\NN) $.
\end{theorem}

We can now give a proof of the theorem \ref{thm2} :

\begin{theorem}
The composition \[ HH^{\bullet}(\mathfrak{u}_0(\mathfrak{g})) \overset{BL}{\to} \bigoplus_{i+j+k= \bullet} H^i(\NN, \wedge^j T \NN)^k \overset{\langle -, Todd(\NN)^{-1/2} \rangle }{\to} \bigoplus_{i+j+k= \bullet} H^i(\NN, \wedge^j T \NN)^k  \] is a ring isomorphism.
\end{theorem}

\begin{proof}
Let us recall how to deduce the isomorphism \ref{thm0} from the theorem \ref{thmBL}. There are isomorphisms $$HH^{\bullet}(\u_0) \overset{\ref{thmBL}}{\cong} \bigoplus_{q+k=\bullet} Ext^q_{\mathrm{Coh}^{\C^*}(\NN \times \NN)}(\mathcal O_{\Delta}, \mathcal O_{\Delta} \otimes z^k) \cong \bigoplus_{i+j+k = \bullet} H^i_{\C^*} (\NN, \wedge^j T\NN \otimes z^k)$$  Here $H^i_{\C^*}(-) = R^i (Hom_{\C^*}(\O_X, -))$. The second isomorphism follows by taking the $\C^*$ invariant on both side of the usual HKR isomorphism (and using corollary \ref{cor3}). Hence we know that $$H^i_{\C^*} (\NN, \wedge^j T\NN \otimes z^k) = (H^i (\NN, \wedge^j T\NN \otimes z^k))^{\C^*} = (H^i (G/B, pr_*\wedge^j T\NN \otimes z^k))^{\C^*}$$ since $pr : \NN \to G/B$ is an affine $\C^*$-morphism. Since the $\C^*$-action is trivial on $G/B$ by proposition \ref{trivial} we obtain an isomorphism  $$(H^i (G/B, pr_*\wedge^j T\NN \otimes z^k))^{\C^*} \cong H^i(G/B, (pr_* \wedge^jT \NN \otimes z^k)[0]) \cong H^i(G/B, (pr_* \wedge^jT \NN)^{-k})$$ Putting everything together we obtain indeed the isomorphism stated in Theorem \ref{thm0} : $$HH^{\bullet}(u_0) \cong \bigoplus_{q+k=\bullet} Ext^q(\mathcal O_{\Delta}, \mathcal O_{\Delta} \otimes z^k) \cong \bigoplus_{i+j+k = \bullet} H^i(\NN, (\wedge^j T\NN)^k)$$  
Now the isomorphism $HH^{\bullet}(\u_0) \cong \bigoplus_{q+k=\bullet} Ext^q(\mathcal O_{\Delta}, \mathcal O_{\Delta} \otimes z^k) $ is multiplicative. Since the Todd class is $T$-invariant by lemma \ref{lemme3}, and that $I_{HKR}$ is $T$-equivariant by lemma \ref{lemme2} it follows that the usual multiplicative HKR theorem restricts to a multiplicative isomorphism $$\bigoplus_{q+k= \bullet } Ext^q(\mathcal O_{\Delta}, \mathcal O_{\Delta} \otimes z^k) \cong \bigoplus_{i+j+k = \bullet} H^i_{\C^*} (\NN, \wedge^j T\NN \otimes z^k)$$ Finally the isomorphism $$ \bigoplus_{i+j+k=\bullet} H^i_{\C^*} (\NN, \wedge^j T\NN \otimes z^k) \cong  \bigoplus_{i+j+k = \bullet} H^i(G/B, (pr_* \wedge^jT \NN)^k) $$ is multiplicative, because for an affine map $q : X \to Y$, the equivalence $\text{Coh}(X) \cong \text{Coh}(Y)-q_* \mathcal O_X$ is monoidal, hence respect the cup-product.  \end{proof}

\subsection{Multiplicative structure of $HC[\tau]$}

We will compute the multiplicative structure of a subalgebra of $z_0(\g)$, already noticed in \cite{LQ}, where the arguments was given at level of sheaf cohomology. 

\begin{lemma}
If $Y = T^*X$ is a cotangent bundle of a smooth algebraic variety $X$, with projection morphism $pr : Y \to X$, then the Todd class of $Y$ is given by the pullback to $Y$ of  \[ \prod_{\alpha,\beta \in C} \frac{-\alpha\beta}{(1 - e^{-\alpha})(1 - e^{-\beta})} \in H\Omega^{\bullet}(X) \] 
where $C$ is the set of Chern roots of $TY$. In particular, this Todd class can be written $1 + pr^*a$ where $a \in H^{\geq 2}(X)$.
\end{lemma}

\begin{proof}
The map $pr : Y \to X$ induces a short exact sequence \[ 0 \to T_{Y/X} \to TY \to pr^*TX \to 0\] 
Moreover, $T_{Y/X} \cong pr^*(T^*X)$. Using the naturality of the Todd class and the additivity on short exact sequences the result easily follows. 
\end{proof}

For example if $X = G/B$ then $C$ is the set of positive roots of $\g$ and the Todd class of $G/B$ is given by $$ Todd(G/B) = \prod_{\alpha \in \Phi^+} \frac{\alpha}{1 - e^{-\alpha}} $$ 

Let $\tau$ be the the generator of $H^0(\NN, End(T\NN) )$ given by the identity map.

\begin{definition}
We define $HC[\tau]$ to be the subalgebra of $z_0(\g)$ generated by the Harish-Chandra center $HC_{\lambda}$ and $\tau$.
\end{definition}

Of course, for singular blocks $\u_{\lambda}(\g)$ there is a similar subalgebra inside $z_{\lambda}(\g)$, also written $HC[\tau]$ if there is no confusion.

\begin{proposition}\label{mainprop}
For any block, the multiplicative structure of $HC[\tau]$ is untwisted. 
\end{proposition}

\begin{proof}
Let us take $U \subset G/P$ an affine open set isomorphic to an affine space. If $pr : \NN_P \to G/P$ is the projection, then $pr^{-1}(U)$ is an affine space with coordinates $(x,y)$. In these coordinates the Poisson bivector field is $\tau = \sum_i \frac{\partial}{\partial x_i} \wedge \frac{\partial}{\partial y_i}$ and the Todd class is $t = 1 + \sum a_{ij} dx_i \wedge dx_j + \text{higher order terms}$. It follows that $(1-t,\tau) = 0$ where $(\cdot, \cdot)$ is the pairing $T^*\NN_P \otimes T\NN_P \to \O_{\NN}$. In the exact sequence $$0 \to pr^*T^*(G/P) \to \NN_P \to pr^*T(G/P) \to 0$$ the left hand-side corresponds to the fiber of $pr$ since $\NN_P = T^*(G/P)$. In particular an element $ \alpha \in HT^{\bullet}(\NN)$ coming from $H^{\bullet}(X, \Omega^{\bullet})$ can we written in local coordinates as a polynomial in $\frac{\partial}{\partial y_i}$. For such a class it is again clear that $(1-t, \alpha) = 0$. In particular it holds for the Harish-Chandra center $HC \cong H^{\bullet}(G/P)$.
\end{proof}

\begin{corollary}
If $G/P$ is a projective space $\P^n$, then $HH^{0}(\u_{\lambda}(\g)) = z_{\lambda}(\g)$ is the polynomial algebra $\C[x,\tau]/(x^a\tau^ b: a+b > n)$.
\end{corollary}

\begin{proof}
It follows immediately from the previous proposition and the fact (proved in \cite{LQ2}) that $HC[\tau] = z_{\lambda}(\g)$ if $G/P_{\lambda}$ is a projective space. 
\end{proof}

\end{document}